\documentclass{article}
\usepackage{amsthm,amsfonts, amsbsy, amssymb,amsmath,graphicx}
\usepackage{graphics}

\newtheorem{thm}{Theorem}

\newtheorem{ex}{Example}
\newtheorem{crl}{Corollary}
\newtheorem{dfn}{Definition}
\newtheorem{prop}{Proposition}
\newtheorem{st}{Statement}

\title{On Free Knots and Links}

\author{Vassily Olegovich Manturov}

 \def\R{{\mathbb R}}
 \def\Z{{\mathbb Z}}

\newcommand{\ZG}{\Z_{2}{\mathfrak{G}}}
\newcommand{\ZGG}{\tilde \Z_{2}{\mathfrak{G}}}

\begin{document}

\maketitle

\abstract{Both classical and virtual knots arise as formal Gauss
diagrams modulo some abstract moves corresponding to Reidemeister
moves. If we forget about both over/under crossings structure and
writhe numbers of knots modulo the same Reidemeister moves, we get a
dramatic simplification of virtual knots, which kills all classical
knots. However, many virtual knots survive after this
simplification.

We construct invariants of these objects and present their
applications to minimality problems of virtual knots as well as some
questions related to graph-links.

One can easily generalize these results for the orientable case and
apply them for solving non-invertibility problems.

The main idea behind these invariants is some geometrical
construction which reduces the general equivalence to the
equivalence only modulo Reidemeister - 2 move.

This paper is a sequel of the paper \cite{FreeKnots}. }

\section{Introduction}

In \cite{FreeKnots}, for a certain class objects (a terrific
simplification of virtual knots) we proved a theorem that the
equivalence questions of some diagrams can often be reduced to the
question of some very simple equivalence (using only Reidemeister 2
moves). To do that, we made difference between two types of
crossings: the ``odd'' ones and the ``even'' ones and created a
diagram-valued invariant of our objects (free knots). For some
diagrams which are ``irreducibly odd'' the value of the invariant
consists of the diagram itself, and it has been considered then only
up to second Reidemeister move.

This has a lot of corollaries (already described here or still to
come in forthcoming papers): for flat virtual knots (for virtual
knots and their generalizations see \cite{KaV}), their
non-triviality, non-equivalence, non-invertibility etc.

However, the main nerve of the construction was the notion of odd
crossing. Roughly speaking, we are taking a Gauss diagram and
forgetting all over/under information and all writhe numbers modulo
formal Reidemeister moves. Odd crossings are precisely those
corresponding to {\em odd chords of the Gauss diagram}, i.e., those
chords which intersect an odd number of other chords.

The notion of odd chord (odd crossing) is closely connected to the
notion of a non-orientable atom: starting from such a four-valent
(framed) graph, we may construct a checkerboard surface (see ahead)
which can be either orientable or non-orientable. The presence of
odd chords precisely means non-oreintability of the atoms.

So, the examples we have constructed and all non-trivial results we
have proved in \cite{FreeKnots} deal with only non-orientable atoms
and (virtual) knots corresponding to them.

By definition, the invariant constructed in \cite{FreeKnots}
``smoothes'' all even crossings, and for a diagram with orientable
atom we get the same value of the invariant as that of the unknot.

So, the simple non-triviality, non-invertibility, non-equivalence,
and minimality results hold for a class of objects with
non-orientable atoms. In particular, it has nothing to do with
classical knots.

The aim of the present paper is to find another ``oddness''
condition for the crossings allowing to use the techniques similar
to the one introduced in \cite{FreeKnots}.

We give  new non-trivial examples.

\section{Basic Notions}

By a {\em $4$-graph} we mean a topological space consisting of
finitely many components, each of which is either a circle or a
finite graph with all vertices having valency four.

A $4$-graph is {\em framed} if for each vertex of it, the four
emanating half-edges are split into two sets of edges called {\em
(formally) opposite}.

A {\em unicursal component} of a $4$-graph is either a free loop
component of it or an equivalence class of edges where two edges
$a$,$b$ are called equivalent if there is a sequence of edges
$a=a_{0},\dots, a_{n}=b$ and vertices $v_{1},\dots, v_{n}$ so that
$a_{i}$ and $a_{i+1}$ are opposite at $v_{i+1}$.

As an example of a free graph one may take the graph of a singular
link.

By a {\em free link} we mean an equivalence class of framed
$4$-valent graphs  modulo the following transformations. For each
transformation we assume that only one fixed fragment  of the graph
is being operated on (this fragment is to be depicted) or some
corresponding fragments of the chord diagram. The remaining part of
the graph or chord diagram are not shown in the picture; the pieces
of the chord diagram not containing chords participating in this
transformation, are depicted by punctured arcs. The parts of the
graph are always shown in a way such that the formal framing
(opposite edge relation) in each vertex coincides with the natural
opposite edge relation taken from  ${\bf R}^{2}$.

The first Reidemeister move is an addition/removal of a loop, see
Fig.\ref{1r}

\begin{figure}
\centering\includegraphics[width=200pt]{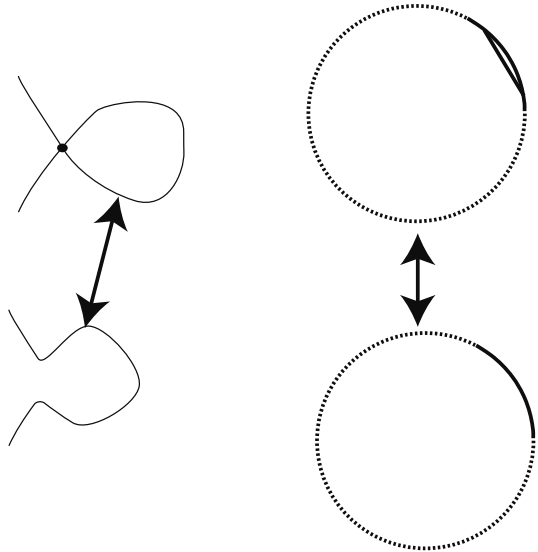}
\caption{Addition/removal of a loop on a graph and on a chord
diagram} \label{1r}
\end{figure}

The second Reidemeister move adds/removes a bigon formed by a pair
of edges which are adjacent in two edges, see Fig. \ref{2r}.

\begin{figure}
\centering\includegraphics[width=300pt]{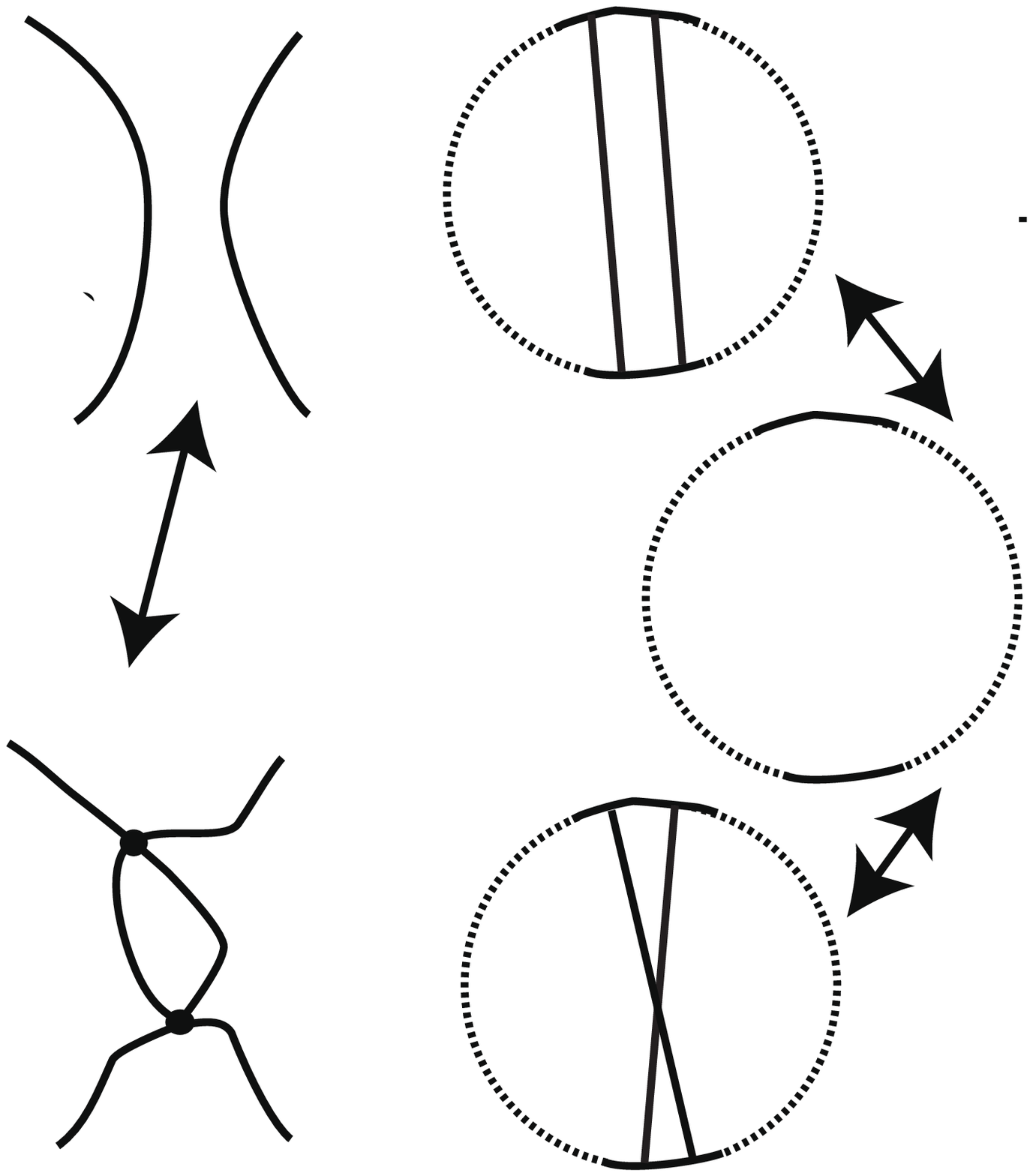} \caption{The second
Reidemeister move and two chord diagram versions of it} \label{2r}
\end{figure}

Note that the second Reidemeister move adding two vertices does not
impose any conditions on the edges it is applied to: we may take any
two two edges of the graph an connect them together as shown in Fig.
\ref{2r} to get two new crossings.

The third Reidemeister move is shown in Fig.\ref{3r}.

\begin{figure}
\centering\includegraphics[width=300pt]{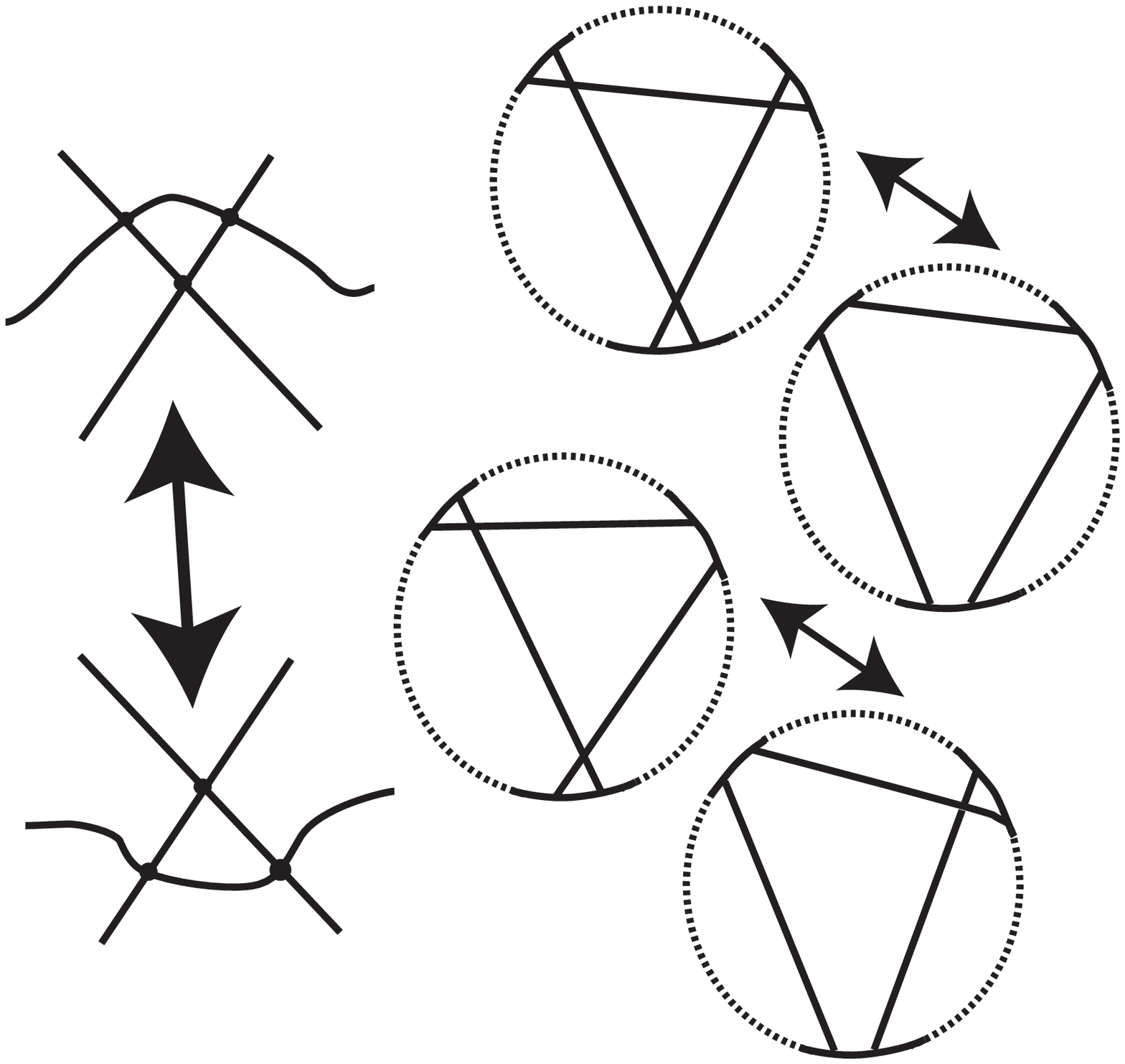} \caption{The third
Reidemeister move and its chord diagram versions} \label{3r}
\end{figure}

Note that each of these three moves applied to a framed graph,
preserves the number of unicursal components of the graph. Thus,
applying these moves to graphs with a unique unicursal cycle, we get
to graphs with a unique unicursal cycle.

A {\em free knot} is a free link with one unicursal component
(obviously, the number of unicursal component of a framed $4$-graph
is preserved under Reidemeister moves).

Free links are closely connected to  {\em flat virtual knots},
i.\,e., with equivalence classes of virtual knots modulo
transformation changing over/undercrossing structure. The latter are
equivalence classes of immersed curves in orientable $2$-surfaces
modulo homotopy and stabilization.

\subsection{Smoothings}

Here we introduce the notion of smoothing, we shall often use in the
sequel.

Let $G$ be a framed graph, let $v$ be a vertex of $G$ with four
incident half-edges $a,b,c,d$, s.t. $a$ is opposite to $c$ and $b$
is opposite to $d$ at $v$.

By {\em smoothing} of $G$ at $v$ we mean any of the two framed
$4$-graphs obtained by removing $v$ and repasting the edges as
$(a,b)$, $(c,d)$ or as $(a,d)$ $(b,c)$, see Fig. \ref{smooth}.

\begin{figure}
\centering\includegraphics[width=200pt]{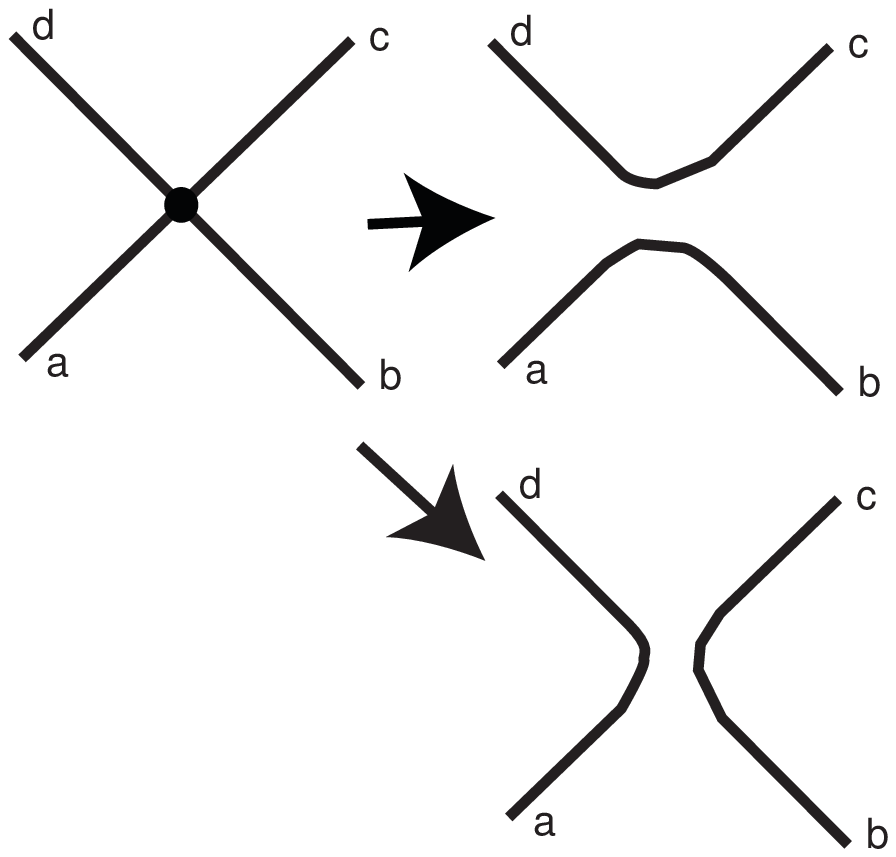} \caption{Two
smoothings of a vertex of for a framed graph} \label{smooth}
\end{figure}

Herewith, the rest of the graph (together with all framings at
vertices except $v$) remains unchanged.

We may then consider further smoothings of $G$ at {\em several}
vertices.

\section{Atoms and orientability. \\ The source-sink condition}

Our further strategy is as follows: in many situations, it is easier
to find {\em links} rather than {\em knots} with desired
non-triviality properties. So, we shall first define a map from free
$1$-compoent links to ${\bf Z}_{2}$-linear combinations of
$2$-component links, and then we shall study the latter by an
invariant similar to that constructed in \cite{FreeKnots}.

Ideologically, the first map is a simplified version of Turaev's
cobracket \cite{Turaev} which establishes a structure of Lie
coalgera on the set of curves immersed in $2$-surfaces (up to some
equivalence, the Lie {\em algebra} structure was introduced by
Goldman in a similar way).

We shall need Turaev's construction (Turaev's $\Delta$) to get a
$2$-component free link from a $1$-component one.

The second map takes a certain state sum for a $2$-component free
link, where we distinguish between two types of crossings, and
smooth only crossings of the first type. What should these ``two
types'' mean, will be discussed later.

In some sense, the invariant $[\cdot]$ of free knots constructed in
\cite{FreeKnots} is a diagrammatic extension of a terrifically
simplified Alexander polynomial (we forget about the variable and
signs taking ${\bf Z}_{2}$-coefficients). The invariant $\{\cdot\}$
suggested in the present paper is in the same sense an extension of
the terrifically simplified Kauffman bracket, but again we use
diagrams as coefficients.

Altogether, these two constructions provide an example of
non-trivial and minimal diagrams of free knots with orientable
atoms.

\begin{dfn}
An {\em atom} (originally introduced by Fomenko, \cite{Fom}) is a
pair $(M,\Gamma)$ consisting of a $2$-manifold $M$ and a graph
$\Gamma$ embedded in $M$ together with a colouring of $M\backslash
\Gamma$ in a checkerboard manner. An atom is called {\em orientable}
if the surface $M$ is orientable. Here $\Gamma$ is called the {\em
frame} of the atom, whence by {\em genus} (atoms and their genera
were also studied by Turaev~\cite{Turg}, and atom genus is also
called the Turaev genus~\cite {Turg}) ({\em Euler characteristic,
orientation}) of the atom we mean that of the surface $M$.
 \end{dfn}

Having  an atom $V$, one can construct a virtual link diagram out of
it as follows. Take a generic immersion of atom's frame into $\R^2$,
for which the formally opposite structure of edges coincides with
the opposite structure induced from the plane.

Put virtual crossings at the intersection points of images of
different edges and restore classical crossings at images of
vertices `as above'. Obviously, since we disregard virtual
crossings, the most we can expect is the well-definiteness up to
detours. However, this allows us to get different virtual link types
from the same atom, since for every vertex $V$ of the atom with four
emanating half-edges $a,b,c,d$ (ordered cyclically on the atom) we
may get two different clockwise-orderings on the plane of embedding,
$(a,b,c,d)$ and $(a,d,c,b)$. This leads to a move called {\em
virtualisation}.

 \begin{dfn}
By a {\em virtualisation} of a classical crossing of a virtual
diagram we mean a local transformation shown in
Fig.~\ref{virtualisation}.
 \end{dfn}

 \begin{figure} \centering\includegraphics[width=200pt]{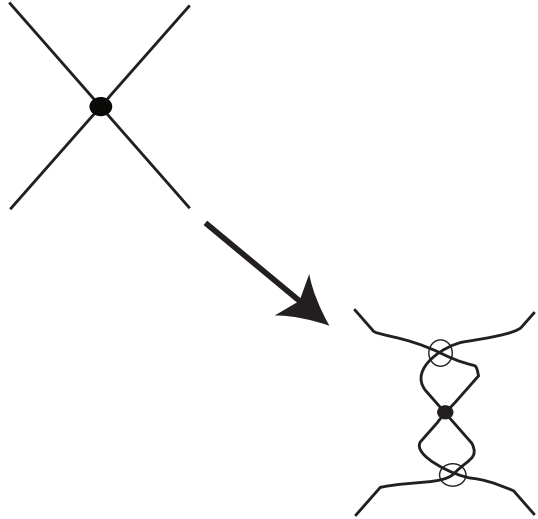}
  \caption{Virtualisation} \label{virtualisation}
 \end{figure}

The above statements summarise as

 \begin{prop}(see, e.g., {\em \cite{MyBook}}).
Let $L_1$ and $L_2$ be two virtual links obtained from the same atom
by using different immersions of its frame. Then $L_1$ differs from
$L_2$ by a sequence of \(detours and\/\) virtualisations.
 \end{prop}

At the level of Gauss diagrams, virtualisation is the move that does
not change the writhe numbers of crossings, but inverts the arrow
directions. So, atoms just keep the information about signs of Gauss
diagrams, but not of their arrows.

A further simplification comes when we want to forget about the
signs and pass to flat virtual links (see also \cite{Vstrings}): in
this case we don't want to know which branch forms an overpass at a
classical crossing, and which one forms an underpass. So, the only
thing we should remember is its frame with opposite edge structure
of vertices (the {\em $A$-structure}). Having that, we take any atom
with this frame and restore a virtual knot up to virtualisation and
crossing change.

The $A$-structure of an atom's frame is exactly a $4$-valent framed
graph.

This perfectly agrees with the fact that {\bf free links are virtual
links modulo virtualization and crossing changes.}

Having a framed $4$-graph, one can consider {\em all atoms} which
can be obtained from it by attaching black and white cells to it. In
fact, it turns out that for a given framed $4$-graph either all such
surfaces are orientable or they are all non-orientable.

To see this, one should introduce the {\em source-sink} orientation.
By a {\em source-sink} orientation of a $4$-valent framed graph we
mean an orientation of all edges of this graph in such a way that
for each vertex some two opposite edges are outgoing, whence the
remaining two edges are incoming.

The following statement is left to the reader as an exercise
\begin{ex}
Let $G$ be a $4$-valent framed graph. Then the following conditions
are equivalent:

1. $G$ admits a source-sink orientation

2. At least one atom obtained from $G$ by attaching black and white
cells is orientable.

3. All atoms obtained from $G$ by attaching black and white cells
are orientable.

Moreover, if $G$ has one unicursal component, then each of the above
conditions is equivalent to the following:

Every chord of the corresponding Gauss diagram $C(D)$ is odd.

\end{ex}

We give two examples: for a planar $4$-valent framed graph we
present a source-sink orientation (left picture, Fig. \ref{lfr}),
and for a non-orientable $4$-valent framed graph (right picture,
Fig.\ref{lfr}, the artefact of immersion is depicted by a virtual
crossing) we see that the source-sink orientation taken from the
left crossing leads to a contradiction for the right crossing.

\begin{figure}
\centering\includegraphics[width=200pt]{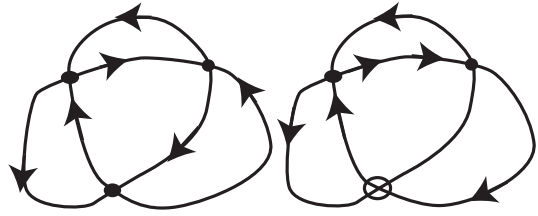} \caption{The
source-target condition} \label{lfr}
\end{figure}

\subsection{The sets ${\ZG}$ and ${\ZGG}$}.

Let ${\mathfrak{G}}$ be the set of all equivalence classes of framed
 graphs with one unicursal component modulo second Reidemeister moves.
Consider the linear space $\ZG$. By ${\ZGG}$ we denote the space of
all equivalence classes of framed graphs (with arbitrarily many
components) by the second Reidemeister move with all graphs with
free loops taken to be zero.

Having a framed $4$-graph, one can consider it as an element of
$\ZG$ or of $\ZGG$. It is natural to try simplifying it: we call a
graph in $\ZG$ {\em irreducible} if no decreasing second
Reidemeister move can be applied to it. We call a graph in $\ZG$
{\em irreducible} if it has no free loops and no decreasing second
Reidemeister move can be applied to it.

The following theorem is trivial

\begin{thm}
Every $4$-valent framed graph $G$ with one unicursal component
considered as an element of $\ZG$ has a unique irreducible
representative, which can be obtained from $G$ by consequtive
application of second decreasing Reidemeister moves.

Every $4$-valent framed graph $G$ considered as an element of $\ZGG$
is either equal to $0$ or has a unique irreducible representative.
In both cases, the reduction can be held by monotonous decreasing of
the diagram by using second Reidemeister move and, if at some point
one gets a free loop, the diagram is equal to zero.

\end{thm}

This allows to recognize elements $\ZG$ and $\ZGG$ easily, which
makes the invariants constructed in the previous subsection
digestable.

In particular, the minimality of a framed $4$-graph in $\ZG$ or
$\ZGG$ is easily detectable: one should just check all pairs of
vertices and see whether any of them can be cancelled by a second
Reidemeister move (or in $\ZG$ one should also look for free loops).

Denote by $\ZGG_{k}$ the subspace of $\ZGG$ generated by
$k$-component free links.

\section{The Turaev cobracket}

There is a simple and fertile idea due to Goldman \cite{Goldman} and
Turaev \cite{Turaev} of transforming two-component curves into
one-component curves and vice versa.

Here we simplify Turaev's idea for our purposes and call it
``Turaev's $\Delta$''.

We shall construct a map from $\ZG$ to $\ZGG_{2}$ as follows.

In fact, to define the map $\Delta$, one may require for a free knot
to be oriented. However, we can do without.

Given a framed $4$-graph $G$. We shall construct an element
$\Delta(G)$ from $\ZGG_{2}$ as follows. For each crossing $c$ of
$G$, there are two ways of smoothing it. One way gives a knot, and
the other smoothing gives a $2$-component link $G_c$. We take the
one giving a $2$-component link and write

\begin{equation}
\Delta(G)=\sum_{c}G_{c}\in \ZGG_{2}
\end{equation}

\begin{thm}
$\Delta(G)$ is a well defined mapping from $\ZG$ to $\ZGG_{2}$
\end{thm}

The proof is standard and follows Turaev's original idea. One should
consider the three Reidemeister move. The first move adds a new
summand which has a free loop (the latter assumed to be trivial in
$\ZGG_{2}$); for the second Reidemeister move we get two new
identical summands, which cancel each other because we are dealing
with ${\bf Z}_{2}$ coefficients. For the third Reidemeister moves
the LHS and the RHS will lead to the summands identical up to second
Reidemeister moves.

{\em We call the conditions described above {\bf the parity
conditions.}}

\section{Two Types of Crossings: \\ Reducing all Reidemeister Moves to
the Second Reidemeister move}

Assume we have a certain class of knot-like objects which are
equivalence classes of {\bf diagrams} modulo {\em three Reidemeister
moves}. Assume for this class of diagrams (e.g. $4$-valent framed
graphs) there is a fixed rule of distinguishing between two types of
crossings (called even and odd) such that:

1) Each crossing taking part in the first Reidemeister move is even,
and after adding/deleting this crossing the parity of the remaining
crossings remains the same.

2) Each two crossings taking part in the second Reidemeister move
are either both odd or both even, and after performing these moves,
the parity of the remaining crossings remains the same.

3) For the third Reidemeister move, the parities of the crossings
which do not take part in the move remain the same.

Moreover, the parities of the three pairs of crossings are the same
in the following sense: there is a natural one-to-one correspondence
between pairs of crossings $A-A',B-B',C-C'$ taking part in the third
Reidemeister move, see Fig. \ref{abc}.

\begin{figure}
\centering\includegraphics[width=200pt]{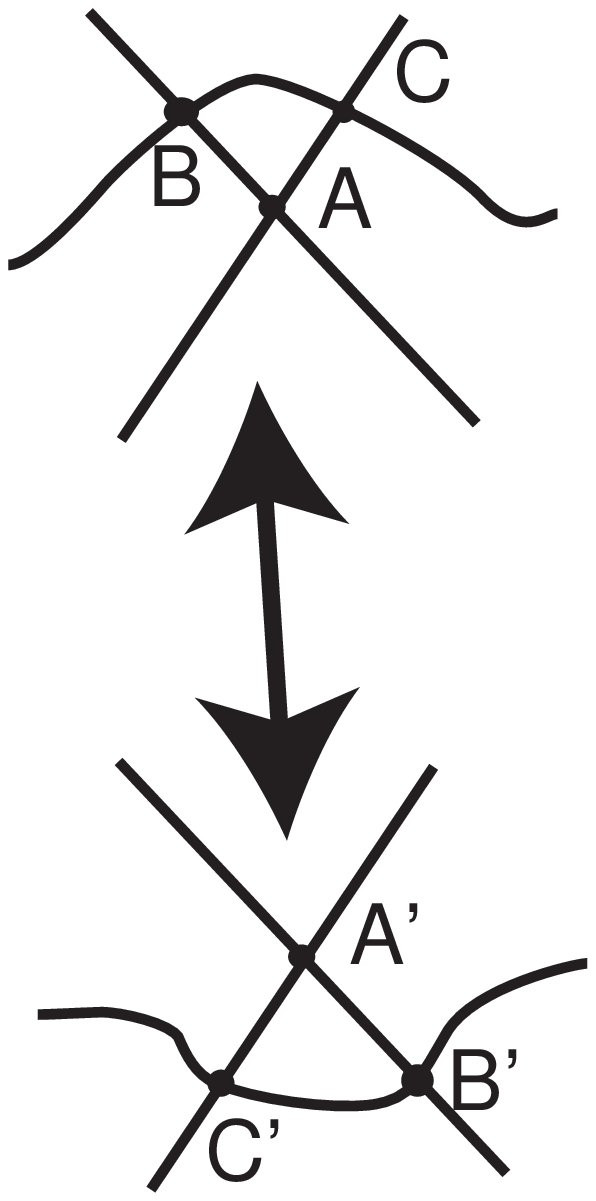}
\label{abc}
\end{figure}

We require that the {\em parity} of $A$ coincides with that of $A'$,
the {\em parity} of $B$ coincides with that of $B'$ and the parity
of $C$ coincides with that of $C'$.

We also require that the number of odd crossings among the three
crossings in question ($A,B,C$) is even (that is, is equal to $2$ or
$0$).

Having such objects with a prescribed rule satisfying the above
properties, one can define an invariant polynomials of our knot-like
objects (to be more precise, this leads to an invariant {\em
mapping} from knot-like objects to a terrifically simpler class of
objects).

In particular, this will lead us to one invariant of graph-knots and
one invariant of graph-links. The first invariant is introduced in
\cite{FreeKnots}.

It counts all {\em rotating circuits}, i.e. circuits going along all
the edges of the graph once and switching from each edge to a
non-opposite edge. Let us be more specific.

If one takes these circuits for a classical knot with an appropriate
signs and weights (powers of $t$), one would get the Alexander
polynomial.

Here, we introduce two ingredients: instead of circuits rotating at
all crossings, we let circles rotate only at ``even'' crossings, and
leave odd crossings as they are. Besides, instead of signs and
powers of $t$, we add some framed graphs modulo relations as
coefficients.

The scheme for the Kauffman bracket polynomial is the same with the
only difference that we count {\em all states} with weights being
either polynomials or chord diagrams: by a state we mean a way of
smoothing of all vertices (or.,resp., all odd vertices). The only
difference between a smoothing and a rotating circuit is that for a
smoothing the number of circles may be arbitrary.

\subsection{The ``Alexander-like'' bracket}

Let us be more specific. We concentrate on free knots and call a
vertex of a $4$-valent graph corresponding to a free knot {\em odd}
if and only if the corresponding chord of the chord diagram
corresponding to the framed graph is odd.

It is left for the reader as an exercise to check the {\bf parity
conditions}. We shall construct a map from free knots to $\ZG$.

Consider the following sum

\begin{equation}
[G]=\sum_{s\;even.,1\; comp} G_{s},
\end{equation}
which is taken over all smoothings in all {\em even} vertices, and
only those summands are taken into account where $G_{s}$ has one
unicursal component.

Thus, if  $G$ has $k$ even vertices, then $[G]$ will consist of at
most $2^{k}$ summands, and if all vertices of $G$ are odd, then we
shall have exactly one summand, the graph $G$ itself.

The ``Alexander-like'' {\em bracket} (to be denoted by $[\cdot ]$)
is defined as follows:

\begin{equation}
[G]=\sum_{s\;even.} G_{s},
\end{equation}

where $G_{s}$ is considered as an element in $\ZG$.

\begin{thm}(\cite{FreeKnots})
The mapping $G\mapsto [G]$ is well defined, i.\,e., $[G]$ does not
depend on the representative of the free knot corresponding to $G$.
\label{mainthm}
\end{thm}

This theorem is proved in \cite{FreeKnots}. The idea behind the
proof relies on the comparison of diagrams obtained from each other
by Reidemeister moves with parity conditions taken into account.

We call a four-valent framed graph having one unicursal component
{\em odd} if all vertices of this graph are odd. We call an odd
graph {\em irreducibly odd} if for every two distinct vertices $a,b$
there exists a vertex $c\notin\{a,b\}$ such that $\langle
a,c\rangle\neq \langle b,c\rangle$.

Theorem \ref{mainthm} yields the following
\begin{crl}
Let  $G$ be an irreducibly odd framed 4-graph with one unicursal
component. Then any representative $G'$ of the free knot
$K_{G}$,generated by $G$, has a smoothing  $\tilde G$ having the
same number of vertices as $G$. In particular, $G$ is a minimal
representative of the free knot $K_{G}$ with respect to the number
of vertices.\label{sld}
\end{crl}

The simplest example of an irreducibly odd graph is depicted in Fig.
\ref{irred}.

\begin{figure}
\centering\includegraphics[width=200pt]{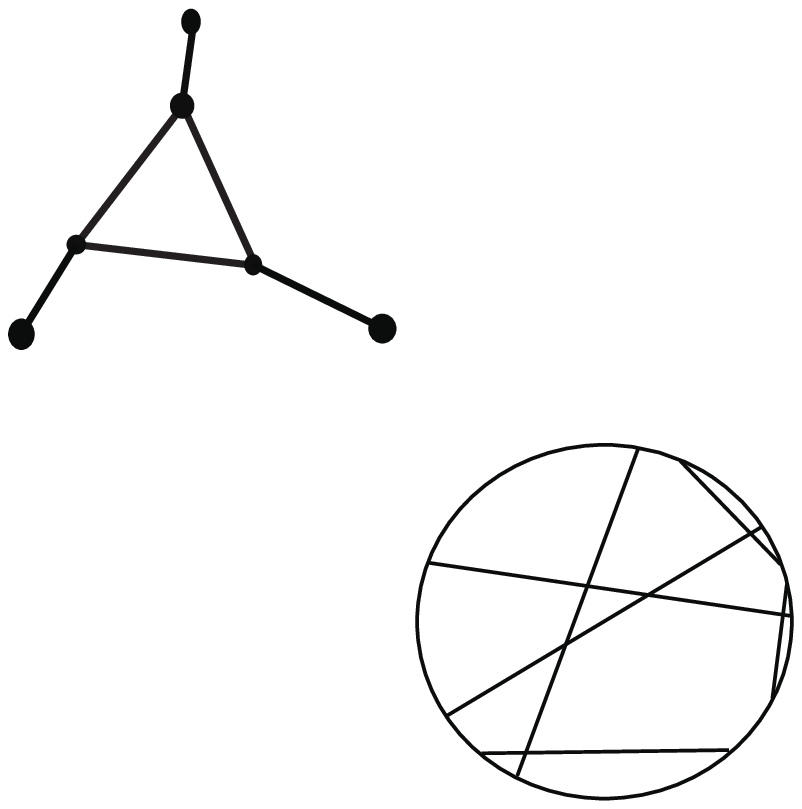} \caption{An
irreducibly odd graph and its chord diagram} \label{irred}
\end{figure}

\section{The Kauffman-like bracket}

Thus,  we have proved that the free knot $K$ having Gauss diagram
shown in Fig.\ref{irred} is minimal: every diagram of this knot has
at least $6$ odd veritces.

The reason is that $[K]$ consists of one diagram representing $K$
itself, and $K$ is not simplifiable in the category $\ZG$.

However, this argument is not applicable to free knots with no odd
crossings. Thus, we give two more examples using ``Kauffman bracket
like'' techniques for other free knots.

Let $K$ be a two-component free considered as an element from
$\ZGG_{2}$. We shall construct a map $\{\cdot\}: K\mapsto \{K\}$
valuend in $\ZGG$ as follows.

Take a framed four-valent graph $G$ representing $K$. By definition,
it has two components. Now, a vertex of $G$ is called {\em odd} if
it is formed by two different components, and {\em even} otherwise.

{\bf The parity conditions can be checked straightforwardly}.

Now, we define

\begin{equation}
\{G\}=\sum_{s} G_{s},
 \label{kbrck}
\end{equation}

where we take the sum over all smoothings of all even vertices, and
consider the smoothed diagrams $K_{s}$ as elements of $\ZGG$. In
particular, we take all elements of $K_{s}$ with free loops to be
zero.

\begin{thm}
The bracket $\{K\}$ is an invariant of two-component free links,
that is, for two graphs $G$ and $G'$ representing the same
two-component free link $K$ we have $\{G\}=\{G'\}$ in $\ZGG$.
\label{mainthm2}
\end{thm}

\begin{proof}
The proof is very similar to that of Theorem \ref{mainthm}. Indeed,
we have to consider two diagrams that differ by a Reidemeister move
and show that the corresponding brackets $\{\cdot\}$ are equal in
$\ZGG$.

Let us check the invariance $[G]\in \ZGG$ under the three
Reidemeister moves.

Let  $G'$ differ from  $G$ by a first Reidemeister move, so that
$G'$ has one vertex more than $G$. By definition this vertex is even
(it is formed by one component), and when calculating $[G']$ this
vertex has to be smoothed in order to get one unicursal curve in
total.

Thus, we have to take only one of two smoothings of the given
vertex, see Fig. \ref{razved}

\begin{figure}
\centering\includegraphics[width=200pt]{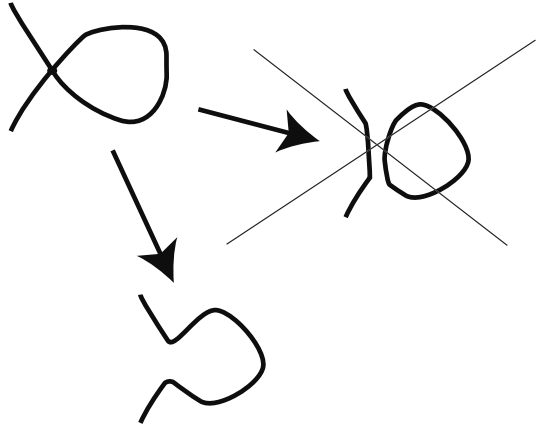} \caption{The two
smoothings --- good and bad --- of a loop} \label{razved}
\end{figure}

Thus there is a natural equivalence between smoothings of $G$ having
one unicursal component, and smoothings of $G'$ with one unicursal
component. Moreover, this equivalence yields a termwise identity
between $[G]$ and $[G']$.

Now, let  $G'$ be obtained from $G$ by a second Reidemeister move
adding two vertices.

These two vertices are either both even or both odd (that is, two
branches belong to the same component of the free link or to
different components).

If both added vertices are odd, then the set of smoothings of $G$ is
in one-to-one correspondence with that of $G'$ and the corresponding
summands for $[G]$ and for $[G']$ differ from each other by a second
Reidemeister move.

If both vertices are odd then one has to consider different
smoothings of these vertices shown in. \ref{razved2}.

\begin{figure}
\centering\includegraphics[width=300pt]{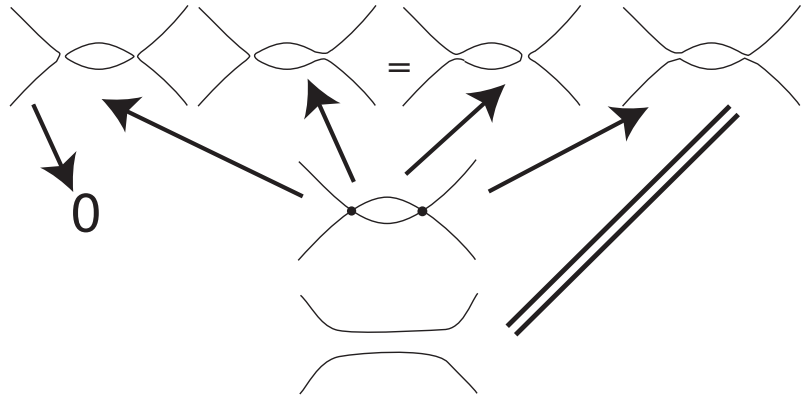}
\caption{Smoothings of two even vertices} \label{razved2}
\end{figure}

The smoothings shown in the upper-left Fig. \ref{razved2}, yield to
free loops, so they do not count in $[G']\in \ZGG$.

The second-type and third-type smoothings (the second and the third
pictures in the top row of \ref{razved2}) give the same impact to
 $\ZG$, thus, they reduce in
$[G']$. Finally, the smoothings corresponding to the upper-right
Fig. \ref{razved2} are in one-to-one correspondence with smoothings
of $G$, thus we have a term-wise equality of terms $[G]$ and those
terms of $[G']$, which are not cancelled by comparing the two middle
pictures.

If $G$ and $G'$ differ by a third Reidemeister move, then the
following two cases are possible: either all vertices taking part in
the third Reidemeister move are even, or two of them are odd and one
is even.

If all the three vertices are even, there are seven types of
smoothings corresponding to $[G]$ (and seven types of smoothings
corresponding to $[G']$): in each of the three vertex we have two
possible smoothings, and one case is ruled out because of a free
loop. When considering $G$, three of these seven cases coincide
(this triple is denoted by $1$), so, in $\ZG$ it remains exactly one
of these two cases. Amongst the smoothings of the diagram $G'$, the
other three cases coincide (they are marked by $2$). Thus, both in
$[G]$ and $[G']$ there are five types of summands marked by
$1,2,3,4,5$.

These five cases are in one-to-one correspondence (see Fig.
\ref{razved31}) and they yield the equality  $[G]=[G']$).

\begin{figure}
\centering\includegraphics[width=200pt]{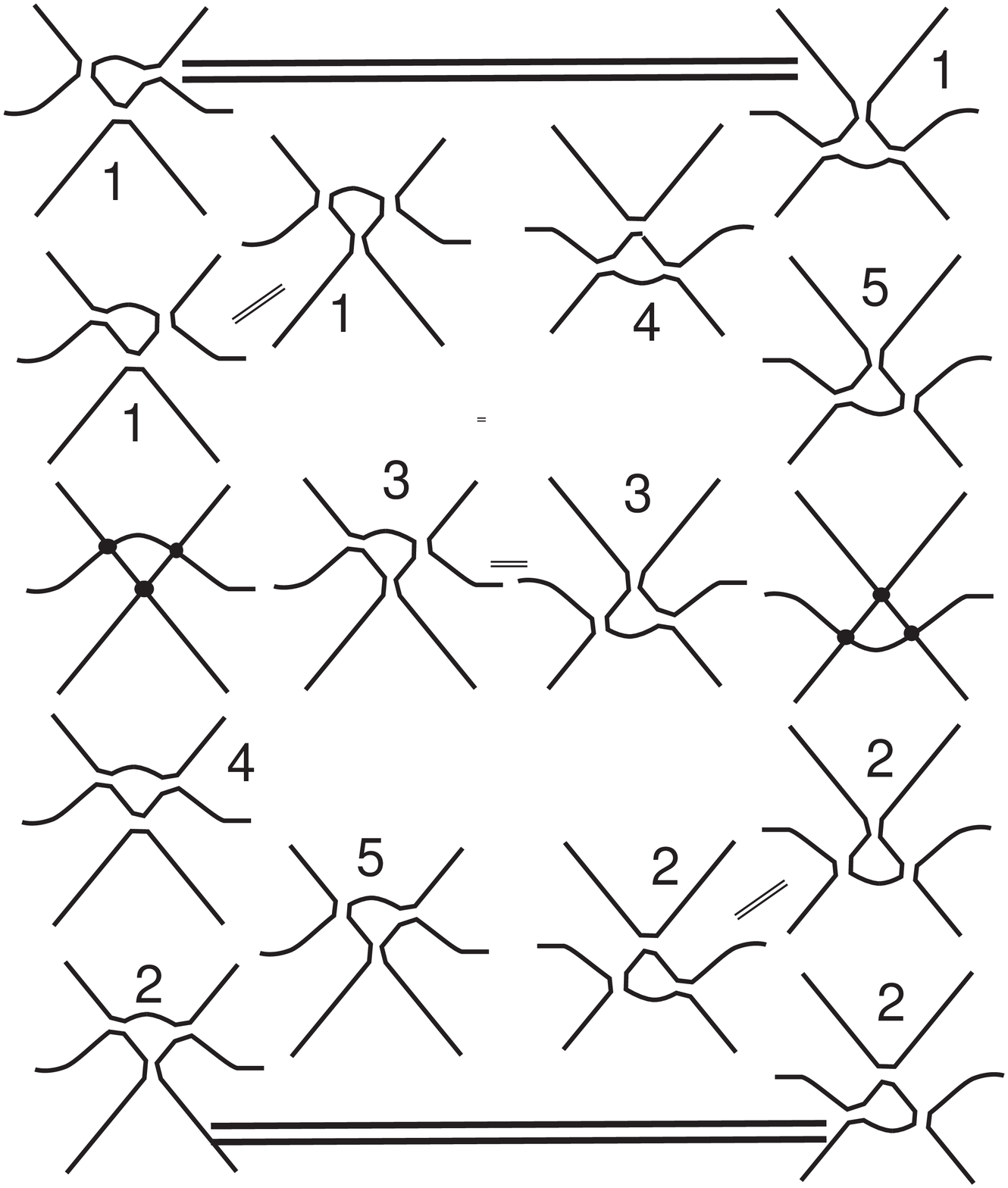}
\caption{Correspondences of smoothings with respect to $\Omega_3$
with three even vertices} \label{razved31}
\end{figure}

If amongst the three vertices taking part in $\Omega_3$ we  have
exactly one even vertices (say $a\to a'$), we get the situation
depicted in Fig. \ref{razved32}.

\begin{figure}
\centering\includegraphics[width=200pt]{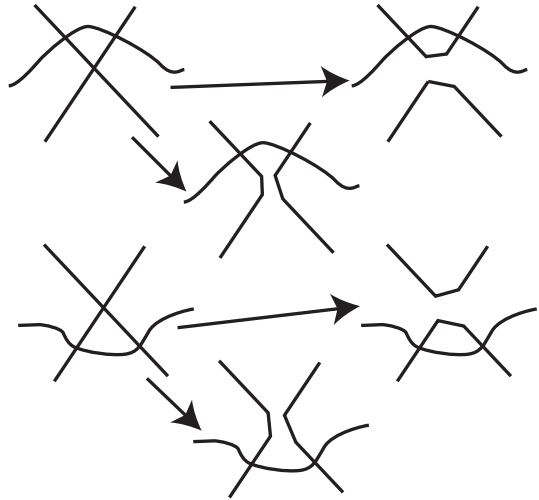}
\caption{Correspondence between smoothings for $\Omega_3$ with one
even vertex} \label{razved32}
\end{figure}

From this figure we see that those smoothings where $a$ (resp.,
$a'$) is smoothed {\em vertically}, give identical summands in $[G]$
and in $[G']$, and those smoothings where $a$ and $a'$ are smoothed
{\em horizontally}, are in one-to-one correspondence for $G$ and
$G'$, and the corresponding summands are obtained by applying two
second Reidemeister moves. This proves that $[G]=[G']$ in $\ZG$.

\end{proof}

We extend $K$ to $\ZGG_{2}$ by linearity.

\section{New Examples}

\begin{st}
The free link $L_1$ shown in Fig. \ref{frlink1} is minimal and the
corresponding atom is orientable.
\end{st}

\begin{figure}
\centering\includegraphics[width=200pt]{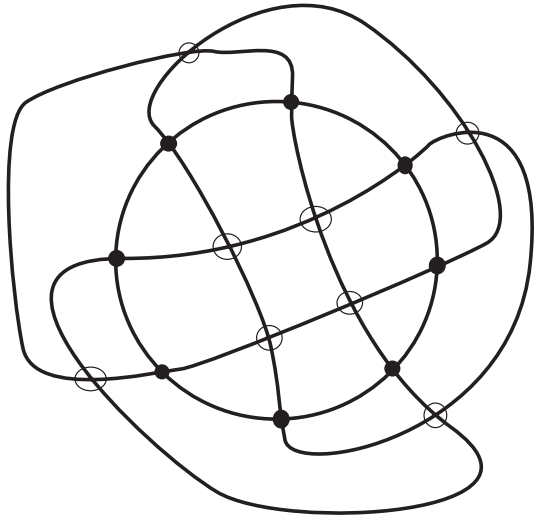}
\caption{Minimal free two-component link} \label{frlink1}
\end{figure}

\begin{proof}
The orientability of any of the corresponding atoms can be checked
straightforwardly: one can easily verify the source-sink condition.

To prove minimality, let us consider the bracket $\{L_1\}$. By
construction, $\{L_1\}$ consists of only one diagram, the diagram
$L_1$ itself. Since $L_1$ is minimal in $\ZG$, we see that for every
link $L'_1$ equivalent to $L_1$ there is a smoothing of $L'_1$ at
some vertices equivalent to $L_1$. So, $L'_1$ has at least eight
crossings.
\end{proof}

\begin{st}
The free knot $K_1$ shown in Fig. \ref{frknot1} is minimal.
\end{st}

\begin{figure}
\centering\includegraphics[width=200pt]{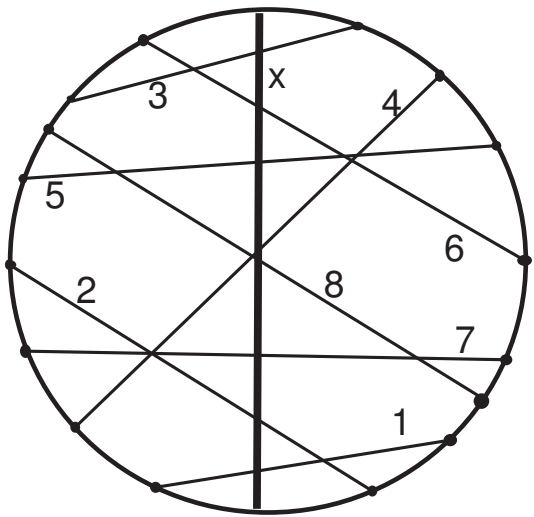} \caption{A
Minimal Free Knot Gauss Diagram} \label{frknot1}
\end{figure}

\begin{proof}
Consider $\Delta(K_1)$. By construction, it consists of nine
summands, each of which representing a two-component free link.
These summands are constructed by smoothing exactly one crossing
(one chord). If we smooth the chord $x$, we get the link $L_1$
depicted above. Thus, $\Delta(K_1)=L_{1}+\sum L_{i}$, where the all
$L_{i}$'s are two-component links.

We claim that none of the links $L_{i}$ is equivalent to $L_1$ as a
free link.

Indeed, the initial Gauss diagram of $K_1$ has $9$ chords, and there
is only one chord $x$ which is linked with any other chords. Thus,
if we smooth the diagram along a chord distinct from $x$, we would
get a $2$-component link, say, $L_i$ with at least one crossing
formed by one and the same component. Thus, by definition of
$\{\cdot\}$, we see that $\{L_{i}\}$ is a sum of diagrams having
strictly less than $8$ crossings. Consequently, $L_{i}\neq L_{1}$.

Thus we have proved that for $\{\Delta(K_1)\}$ has exactly one
diagram with minimal crossing number $8$. Taking into account the
invariance of $\{\cdot\}$, we see that $K_1$ has at least $9$
crossings.

Moreover, we have indeed proved that every diagram of $\Delta$ has
at least one smoothing equivalent to $L_1$.

\end{proof}

\section{Post Scriptum. Odds and Ends.}

In \cite{FreeKnots} we gave an example of a looped graph (or
graph-knot,  for definition see \cite{TZ,IM}) which has no
realisable representative, i.\,e. which is not equivalent to any
virtual knot. However, this graph was consisting of only {\em odd
vertices}, i.\,e., each vertex has odd valency.

Below we present an example of a looped graph with all vertices of
even valency, which has no realizable representative.

To do that, let us analyse the example shown in Fig. \ref{frknot1}.
This Gauss diagram has all even chords (this guarantees the
orientability of the corresponding atom). There is exactly one chord
which is linked with any of the remaining chords: the chord $x$.
This guarantees that the smoothing at $x$ gives a $2$-component link
where every crossing is formed by $2$ components, whence the
smoothing at any other crossing gives a $2$ components gives a
$2$-component link having at least one crossing formed by one
component. This means, that if we then apply the bracket
$\{\cdot\}$, these remaining diagrams will lead to diagrams with
strictly smaller number of crossings (than eight).

Finally, there is no room to perform the second decreasing
Reidemeister move to the graph obtained by smoothing along $x$. This
is guaranteed by the fact that in the initial chord diagram there is
no pair of chords $a,b$ both distinct from $x$ and such that

For reader's convenience, the intersection graph of this chord
diagram looks as shown in Fig. \ref{xx}.

\begin{figure}
\centering\includegraphics[width=200pt]{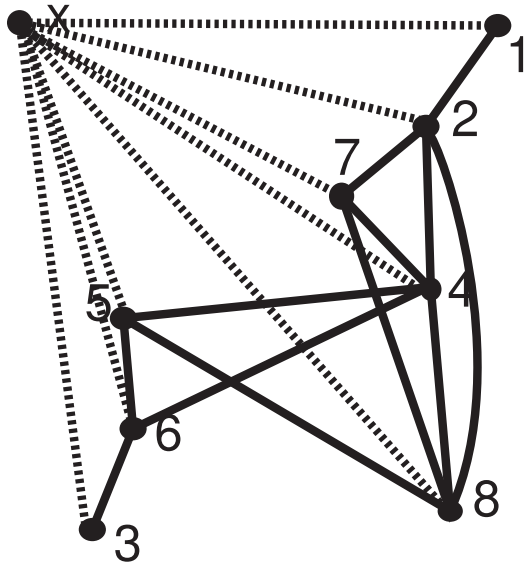} \caption{The
intersection graph of the minimal chord diagram} \label{xx}
\end{figure}

Quite analogously, one considers the looped graph in the sense
\cite{TZ} with ``Gauss diagram intersection graph'' given by the
non-realizable graph shown in Fig. \ref{xxx}.

\begin{figure}
\centering\includegraphics[width=200pt]{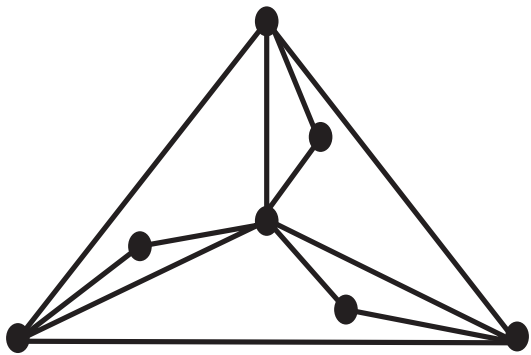} \caption{The
intersection graph a non-realizable knot} \label{xxx}
\end{figure}

All vertices of this diagram have even valency, and there is exactly
one vertex connected to all the remaining vertices. Thus, applying
Turaev's $\Delta$ to it, one gets $7$ $2$-component diagrams:
$A+\sum_{i}B_{i}$, exactly one of which has all ``even crossings'',
that is, only for one diagram $A$ every crossing belongs to both
components of the two component link. For any other diagram $B_{i},
i=1,\dots, 6$, there is at least one crossing formed by branches of
the same component.

So, the diagram $A$ is not cancelled by any of $B_{i}$'s. One easily
checks that $A$ is non-realizable, thus, the looped graph with Gauss
diagram shown in Fig. \ref{xxx} has no realizable representative.

\end{document}